\documentclass[a4paper]{amsart}
\usepackage[foot]{amsaddr}
\usepackage[utf8]{inputenc}
\usepackage[T1]{fontenc}
\usepackage{lmodern}
\usepackage{amsmath}
\usepackage{amsfonts}
\usepackage{amssymb}
\usepackage{amsthm}
\usepackage{thmtools}
\usepackage{mathtools}
\usepackage{tikz-cd}
\usepackage{multirow}
\usepackage[pdftex,
            pdfauthor={Sven Möller},
            pdftitle={Orbifold Vertex Operator Algebras and the Positivity Condition},
            pdfsubject={},
            pdfkeywords={},
            pdfproducer={},
            pdfcreator={}]{hyperref}

\theoremstyle{plain}
\newtheorem{thm}{Theorem}[section]
\newtheorem{prop}[thm]{Proposition}
\newtheorem{cor}[thm]{Corollary}

\newtheorem{conj}[thm]{Conjecture}
\newtheorem*{thm*}{Theorem}

\theoremstyle{definition}

\newtheorem*{defi*}{Definition}

\newtheorem*{nota*}{Notation}

\newcommand{\Q}{\mathbb{Q}}
\newcommand{\Z}{\mathbb{Z}}
\newcommand{\Ns}{\mathbb{Z}_{>0}}
\newcommand{\N}{\mathbb{Z}_{\geq0}}
\newcommand{\C}{\mathbb{C}}
\newcommand{\R}{\mathbb{R}}
\renewcommand{\H}{\mathbb{H}}

\newcommand{\tr}{\operatorname{tr}}
\renewcommand{\i}{\mathrm{i}}
\newcommand{\e}{\mathrm{e}}

\newcommand{\Aut}{\operatorname{Aut}}

\newcommand{\rk}{\operatorname{rk}}

\newcommand{\voa}{vertex operator algebra}

\newcommand{\VOA}{Vertex Operator Algebra}

\newcommand{\fpvosa}{fixed-point vertex operator subalgebra}

\newcommand{\vac}{\textbf{1}}
\newcommand{\ch}{\operatorname{ch}}

\newcommand{\id}{\operatorname{id}}

\newcommand{\eps}{\varepsilon}

\newcommand{\SLZ}{\operatorname{SL}_2(\mathbb{Z})}

\newcommand{\ee}{\mathfrak{e}}

\newcommand{\hh}{\mathfrak{h}}
\newcommand{\h}{{L_\C}}

\renewcommand{\S}{\mathcal{S}}

\newcommand{\qdim}{\operatorname{qdim}}

\newcommand{\strathol}{strongly rational, holomorphic}
\newcommand{\strat}{strongly rational}

\newlength{\myl}
\begin{document}

\title[Orbifold Vertex Operator Algebras and the Positivity Condition]{Orbifold Vertex Operator Algebras and\\ the Positivity Condition}
\author[Sven Möller]{Sven Möller}
\address{Rutgers University, Piscataway, NJ, United States of America}
\email{\href{mailto:math@moeller-sven.de}{math@moeller-sven.de}}
\dedicatory{Für Heike, in Erinnerung.}

\begin{abstract}
In this note we show that the irreducible twisted modules of a holomorphic, $C_2$-cofinite \voa{} $V$ have $L_0$-weights at least as large as the smallest $L_0$-weight of $V$. Hence, if $V$ is of CFT-type, then the twisted $V$-modules are almost strictly positively graded. This in turn implies that the \fpvosa{} $V^G$ for a finite, solvable group of automorphisms of $V$ almost satisfies the positivity condition. These and some further results are obtained by a careful analysis of Dong, Li and Mason's twisted modular invariance.
\end{abstract}

\maketitle

\setcounter{tocdepth}{1}
\tableofcontents

\section*{Introduction}
This note is concerned with \voa{}s and their representation theory. For an introduction to these topics we refer the reader to \cite{FLM88,FHL93,LL04}. Just note that for the definition of twisted modules we follow the modern sign convention (see, e.g.\ \cite{DLM00} as opposed to \cite{Li96}).

Let $V$ be a \voa{}. Recall that $V$ has a $\Z$-grading (by $L_0$-eigenvalues) $V=\bigoplus_{k\in\Z} V_k$ with $\dim(V_k)<\infty$ for all $k\in\Z$ and $\dim(V_k)=0$ for $k\ll 0$. The smallest $k\in\Z$ such that $\dim(V_k)>0$ is called the \emph{conformal weight} of $V$ and denoted by $\rho(V)$.

Similarly, the $L_0$-grading of an irreducible $V$-module $W$ takes values in $\rho(W)+\N$ for some \emph{conformal weight} $\rho(W)\in\C$ with $\dim(W_{\rho(W)})>0$.

Often, \voa{}s are assumed to be of \emph{CFT-type}, in which case the conformal weight $\rho(V)=0$ and $\dim(V_0)=1$. Since the definition of a \voa{} includes a non-zero \emph{vacuum vector} $\vac\in V_0$, this means that in a \voa{} of CFT-type the vacuum vector is the up to scalar unique vector of smallest $L_0$-eigenvalue.

Naturally extending this property to all the modules of $V$ we arrive at the following definition.\footnote{Indeed, if $V$ is rational and of CFT-type, then the positivity condition implies that the vacuum vector is the up to scalar unique vector of smallest $L_0$-eigenvalue in all $V$-modules.}
\begin{defi*}[Positivity Condition]
Let $V$ be a simple \voa{}. $V$ is said to satisfy the \emph{positivity condition} if for all irreducible $V$-modules $W$, $\rho(W)\in\R_{>0}$ if $W\not\cong V$ and $\rho(V)=0$.
\end{defi*}
The positivity condition is a necessary assumption for some important results concerning the relationship between simple currents, quantum dimensions and the $S$-matrix (see \cite{DJX13}, e.g.\ Proposition~4.17) and \voa{}s with group-like fusion, i.e.\ \voa{}s whose irreducible modules are all simple currents (see Section~2.2 in \cite{Moe16} and Section~3 in \cite{EMS15}). Some of these aspects will be briefly discussed in Section~\ref{sec:qdim}.

In this text we study the positivity condition for certain \fpvosa{}s, i.e.\ \voa{}s $V^G$ where $V$ is a holomorphic \voa{} and $G$ a finite group of automorphisms of $V$.

As always, when making non-trivial statements about abstract \voa{}s it is necessary to restrict to a subclass of suitably regular \voa{}s. We follow \cite{DM04b} and call a \voa{} \emph{\strat{}} if it is rational (as defined in \cite{DLM97}, for example), $C_2$-cofinite, self-contragredient (or self-dual) and of CFT-type. A \voa{} $V$ is called \emph{holomorphic} if it is rational and the only irreducible $V$-module is $V$ itself (implying that $V$ is simple and self-contragredient).

\subsection*{Acknowledgements}
The author would like to thank Nils Scheithauer and Jethro van Ekeren for helpful discussions. The author is grateful to Toshiyuki Abe for organising the workshop ``Research on algebraic combinatorics and representation theory of 
finite groups and vertex operator algebras'' at the RIMS in Kyoto in December 2018. This research note was prepared for the proceedings of this workshop (to appear in the RIMS Kôkyûroku series).

\section{Conjecture and Main Result}

Let $V$ be a \strathol{} \voa{} (which implies by the modular invariance result of \cite{Zhu96} that the central charge $c$ of $V$ is in $8\Ns$). It is shown in \cite{DLM00} that for any automorphism $g$ of $V$ of finite order $n\in\N$, $V$ possesses an up to isomorphism unique irreducible $g$-twisted $V$-module, denoted by $V(g)$. Note that $V(\id)\cong V$. Moreover, the conformal weight $\rho(V(g))$ of $V(g)$ is rational. In fact, it is shown in \cite{EMS15,Moe16} that $\rho(V(g))\in\frac{1}{n^2}\Z$. The $L_0$-eigenvalues of $V(g)$ lie in $\rho(V(g))+\frac{1}{n}\N$.

We hypothesise that the possible values of $\rho(V(g))$ can be further constrained (see also Conjecture~\ref{conj:main2} for a more general statement):
\begin{conj}[Main Conjecture]\label{conj:main}
Let $V$ be a \strathol{} \voa{} and $g\neq\id$ an automorphism of $V$ of finite order. Then $\rho(V(g))>0$.
\end{conj}
This is closely related to the positivity condition for \fpvosa{}s.
\begin{prop}\label{prop:cor}
Let $V$ be a \strathol{} \voa{} such that Conjecture~\ref{conj:main} is satisfied for $V$. Let $G\leq\Aut(V)$ be a finite, solvable group of automorphisms of $V$. Then $V^G$ satisfies the positivity condition.
\end{prop}
The proof of this statement is immediate with the following result:
\begin{prop}
Let $V$ be a simple, \strat{} \voa{} and $G$ a finite, solvable group of automorphisms of $V$. Then every irreducible $V^G$-module appears as a $V^G$-submodule of the $g$-twisted $V$-module for some $g\in G$.
\end{prop}
\begin{proof}
The idea of the proof was first written down by Miyamoto (see proof of Lemma~3 in \cite{Miy10}, where only the case of cyclic $G$ is covered). For the stated generality we need Theorem~3.3 in \cite{DRX15}. Its hypotheses (see also Remark 3.4 in \cite{DRX15}) are satisfied since $V^G$ is again \strat{} by \cite{Miy15,CM16}. For this, the group $G$ has to be solvable.
\end{proof}
\begin{proof}[Proof of Proposition~\ref{prop:cor}]
Let $W$ be an irreducible $V^G$-module. If $W$ appears as a submodule of $V(g)$ for some non-trivial $g\in G$, then $\rho(W)>0$. So, let $W$ be an irreducible $V^G$-submodule of $V$. Since any automorphism $g\in\Aut(V)$ by definition fixes the vacuum vector $\vac$, $\rho(V^G)=0$ (note that $V^G$ is again simple, i.e.\ irreducible) but no other irreducible $V^G$-submodule $W$ of $V$ can contain a vector of weight 0 because $V_0=\C\vac$ since $V$ is of CFT-type.
\end{proof}

In the following we will prove a statement that is only slightly weaker than Conjecture~\ref{conj:main}. We are not aware of any counterexamples to the conjecture but would be interested in learning of such. Section~\ref{sec:qdim} explores some of the consequences of Conjecture~\ref{conj:main} not being satisfied.
\begin{thm}[Main Result]\label{thm:main}
Let $V$ be a \strathol{} \voa{} and $g$ an automorphism of $V$ of finite order. Then $\rho(V(g))\geq 0$ and $\dim(V(g)_0)=1$ if $\rho(V(g))=0$.
\end{thm}

As a corollary we can slightly strengthen the statement of Theorem~5.11 in \cite{EMS15}:
\begin{cor}
Let $V$ be a \strathol{} \voa{} and $g$ an automorphism of $V$ of finite order $n\in\Ns$. Then
\begin{equation*}
\rho(V(g))\in\frac{1}{n^2}\N.
\end{equation*}
\end{cor}
This is a statement in the spirit of Theorems 1.2~(ii) and 1.6~(i) in \cite{DLM00} but considerably stronger.

\section{Proof}
In this section we prove Theorem~\ref{thm:main}, the main result of this text. In fact, we show a slightly stronger statement, only requiring the \voa{} $V$ to be holomorphic and $C_2$-cofinite. This amounts to dropping the requirement that $V$ be of CFT-type, which is included in strong rationality.

We recall important results from Dong, Li and Mason's twisted modular invariance paper \cite{DLM00}. Let $V$ be a holomorphic, $C_2$-cofinite \voa{} and let $g\in\Aut(V)$ be of finite order $n\in\Ns$. Then there is an up to isomorphism unique irreducible $g$-twisted $V$-module $V(g)$ and the formal power series
\begin{equation*}
Z_{\id,g}(q):=\tr_{V}gq^{L_0-c/24}=q^{\rho(V)-c/24}\sum_{k\in\N}\tr_{V_{\rho(V)+k}}gq^k
\end{equation*}
and
\begin{equation*}
Z_{g,\id}(q):=\tr_{V(g)}q^{L_0-c/24}=q^{\rho(V(g))-c/24}\sum_{k\in\frac{1}{n}\N}\dim(V(g)_{\rho(V(g))+k})q^k,
\end{equation*}
called \emph{twisted trace functions}, converge to holomorphic functions in $\tau$ on the upper half-plane $\H$ upon identifying $q=\e^{2\pi\i\tau}$, which we shall always do in the following. Moreover, the $S$-transformation of $Z_{\id,g}$ is proportional to $Z_{g,\id}$, i.e.\
\begin{equation}\label{eq:dlm}
Z_{\id,g}(S.\tau)=\lambda_{\id,g}Z_{g,\id}(\tau)
\end{equation}
for some non-zero $\lambda_{\id,g}\in\C$ where $S.\tau=-1/\tau$. Here $S=\left(\begin{smallmatrix}0&-1\\1&0\end{smallmatrix}\right)$ and we let $\SLZ$ act on $\H$ via the Möbius transformation. Note that both the central charge $c$ and the conformal weight $\rho(V(g))$ are rational \cite{DLM00}.

\begin{thm}\label{thm:main2}
Let $V$ be a holomorphic, $C_2$-cofinite \voa{} and $g$ an automorphism of $V$ of finite order. Then $\rho(V(g))\geq\rho(V)$ and
\begin{equation*}
|\lambda_{\id,g}|\dim(V(g)_{\rho(V)})\leq\dim(V_{\rho(V)})
\end{equation*}
if $\rho(V(g))=\rho(V)$.
\end{thm}
\begin{proof}
We study equation~\eqref{eq:dlm}, specialised to $\tau=\i t$ for $t\in\R_{>0}$
\begin{align}\label{eq:dlm2}
\begin{split}
&\e^{2\pi(c/24-\rho(V))/t}\sum_{k\in\N}\tr_{V_{\rho(V)+k}}g\e^{-2\pi k/t}\\
&=\lambda_{\id,g}\e^{2\pi t(c/24-\rho(V(g)))}\sum_{k\in\frac{1}{n}\N}\dim(V(g)_{\rho(V(g))+k})\e^{-2\pi tk}.
\end{split}
\end{align}
Rearranging some factors we arrive at
\begin{equation*}
l_g(t)=r_g(t)
\end{equation*}
with the functions $l_g,r_g:\R_{>0}\to\R_{>0}$ defined as
\begin{align*}
l_g(t)&:=\frac{1}{\lambda_{\id,g}}\e^{2\pi (c/24-\rho(V))(1/t-t)}\sum_{k\in\N}\tr_{V_{\rho(V)+k}}g\e^{-2\pi k/t},\\
r_g(t)&:=\e^{2\pi t(\rho(V)-\rho(V(g)))}\sum_{k\in\frac{1}{n}\N}\dim(V(g)_{\rho(V(g))+k})\e^{-2\pi tk}.
\end{align*}
It is obvious that $r_g$ and hence $l_g$ takes values in $\R_{>0}$.

We observe that
\begin{align}\label{eq:leq}
\begin{split}
l_g(t)&=|l_g(t)|\leq\frac{1}{|\lambda_{\id,g}|}\e^{2\pi (c/24-\rho(V))(1/t-t)}\sum_{k\in\N}|\tr_{V_{\rho(V)+k}}g|\e^{-2\pi k/t}\\
&\leq\frac{1}{|\lambda_{\id,g}|}\e^{2\pi (c/24-\rho(V))(1/t-t)}\sum_{k\in\N}\dim(V_{\rho(V)+k})\e^{-2\pi k/t}\\
&=\frac{\lambda_{\id,\id}}{|\lambda_{\id,g}|}l_{\id}(t)
\end{split}
\end{align}
where we used that for a finite-order automorphism $g$ on a $\C$-vector space of dimension $m$, $|\tr(g)|=|\sum_{i=1}^m\mu_i|\leq\sum_{i=1}^m|\mu_i|=\sum_{i=1}^m 1=m$ where the $\mu_i$ are the eigenvalues of $g$ counted with algebraic multiplicities.

Furthermore, observe that
\begin{equation*}
\sum_{k\in\frac{1}{n}\N}\dim(V(g)_{\rho(V(g))+k})\e^{-2\pi tk}\stackrel{t\to\infty}{\longrightarrow}\dim(V(g)_{\rho(V(g))}),
\end{equation*}
which is non-zero by the definition of the conformal weight. This follows for example from the monotone convergence theorem for non-increasing sequences of functions applied to the counting measure on $\N$. Hence
\begin{equation*}
r_g(t)\stackrel{t\to\infty}{\longrightarrow}\begin{cases}\infty&\text{if }\rho(V(g))<\rho(V),\\\dim(V(g)_{\rho(V)})&\text{if }\rho(V(g))=\rho(V),\\0&\text{if }\rho(V(g))>\rho(V),\end{cases}
\end{equation*}
which in the special case of $g=\id$ becomes
\begin{equation*}
r_{\id}(t)\stackrel{t\to\infty}{\longrightarrow}\dim(V_{\rho(V)}).
\end{equation*}
Using equation~\eqref{eq:leq} and comparing $l_g(t)=r_g(t)$ with $l_{\id}(t)=r_{\id}(t)$ in the limit of $t\to\infty$ it is apparent that $\rho(V(g))\geq\rho(V)$ and
\begin{equation*}
\dim(V(g)_{\rho(V)})\leq\frac{\lambda_{\id,\id}}{|\lambda_{\id,g}|}\dim(V_{\rho(V)})
\end{equation*}
if $\rho(V(g))=\rho(V)$.

Finally, it follows from equation~\eqref{eq:dlm} for $g=\id$ (this is of course just an instance of Zhu's untwisted modular invariance \cite{Zhu96}) and $S^2.\tau=\tau$ that $\lambda_{\id,\id}^2=1$. Moreover, equation~\eqref{eq:dlm2} for $g=\id$ implies $\lambda_{\id,\id}\in\R_{\geq0}$ so that $\lambda_{\id,\id}=1$, which completes the proof.
\end{proof}

\begin{proof}[Proof of Theorem~\ref{thm:main}]
If $V$ is of CFT-type, then $\rho(V)=0$ and $\dim(V_0)=1$. Moreover, it is shown in Proposition~5.5 of \cite{EMS15} (see also Lemma~1.11.2 in \cite{Moe16} and Section~\ref{sec:further} below) that $\lambda_{\id,g}=1$.
This proves the assertion.
\end{proof}

With Theorem~\ref{thm:main2} in mind, it seems natural to generalise Conjecture~\ref{conj:main} to the situation where CFT-type is not assumed:
\begin{conj}\label{conj:main2}
Let $V$ be a holomorphic, $C_2$-cofinite \voa{} and $g\neq\id$ an automorphism of $V$ of finite order. Then $\rho(V(g))>\rho(V)$.
\end{conj}

\section{Further Results}\label{sec:further}
Studying equation \eqref{eq:dlm2} in the limit $t\to 0$ rather than $t\to\infty$ we obtain:
\begin{thm}\label{thm:trr}
Let $V$ be a holomorphic, $C_2$-cofinite \voa{} and $g$ an automorphism of $V$ of finite order. Then
\begin{equation*}
\frac{1}{\lambda_{\id,g}}\tr_{V_{\rho(V)+k}}g\in\R
\end{equation*}
for all $k\in\N$ and
\begin{equation*}
\frac{1}{\lambda_{\id,g}}\tr_{V_{\rho(V)}}g\in\R_{\geq 0}.
\end{equation*}
\end{thm}
If $V$ is in addition of CFT-type, then we obtain:
\begin{cor}
Let $V$ be a \strathol{} \voa{} and $g$ an automorphism of $V$ of finite order. Then
\begin{equation*}
\tr_{V_k}g\in\R
\end{equation*}
for all $k\in\N$ and
\begin{equation*}
\lambda_{\id,g}\in\R_{>0}.
\end{equation*}
\end{cor}
The second part of this statement is proved in Proposition~5.5 of \cite{EMS15} (see also Lemma~1.11.2 in \cite{Moe16}) and used to show that in fact $\lambda_{\id,g}=1$.
\begin{proof}
If $V$ is of CFT-type, then $\rho(V)=0$ and $V_0=\C\vac$. Since any automorphism $g$ fixes the vacuum vector $\vac$ by definition, $\tr_{V_0}g=1$. This proves the second assertion and consequently also the first one.
\end{proof}
\begin{proof}[Proof of Theorem~\ref{thm:trr}]
We again study equation~\eqref{eq:dlm2} but now distribute the factors slightly differently to obtain
\begin{equation*}
L_g(t)=R_g(t)
\end{equation*}
with the functions $L_g,R_g:\R_{>0}\to\R_{>0}$ defined as
\begin{align*}
L_g(t)&:=\frac{1}{\lambda_{\id,g}}\sum_{k\in\N}\tr_{V_{\rho(V)+k}}g\e^{-2\pi k/t},\\
R_g(t)&:=\e^{2\pi (c/24)(t-1/t)}\e^{2\pi\rho(V)/t}\e^{-2\pi t\rho(V(g))}\sum_{k\in\frac{1}{n}\N}\dim(V(g)_{\rho(V(g))+k})\e^{-2\pi tk}.
\end{align*}
Note that $L_g$ takes values in $\R_{>0}$ since $R_g$ does.

Noting again that $|\tr_{V_{\rho(V)+k}}g|\leq\dim(V_{\rho(V)+k})$ and that
\begin{equation*}
\sum_{k\in\N}\dim(V_{\rho(V)+k})\e^{-2\pi k/t}
\end{equation*}
converges, we may apply the dominated convergence theorem for series and interchange the infinite summation with the process of taking the limit $t\to 0$. This shows that
\begin{equation*}
L_g(t)\stackrel{t\to 0}{\longrightarrow}\frac{1}{\lambda_{\id,g}}\tr_{V_{\rho(V)}}g.
\end{equation*}
This has to be in $\R_{\geq0}$ since $R_g(t)$ is for all $t$ and $\R_{\geq0}$ is closed in $\C$, which proves the second assertion.

Now subtract the summand for $k=0$ from $L_g(t)$ (which is in $\R$ as we just showed) and multiply by $\e^{2\pi/t}$ to see in the limit $t\to 0$ that the summand for $k=1$ is in $\R$. Repeating this process for all $k\in\N$ proves the first assertion.
\end{proof}

\section{Examples}
In this section we study the validity of Conjecture~\ref{conj:main} for examples of holomorphic \voa{}s with some of their twisted modules explicitly known.

\subsection{Permutation Orbifolds}
Let $V$ be a \voa{} of central charge $c$. Then an element $g$ of the symmetric group $\mathcal{S}_k$ acts in the obvious way on the tensor-product \voa{} $V^{\otimes k}$ of central charge $kc$ for $k\in\Ns$. In \cite{BDM02} the authors describe the $g$-twisted $V^{\otimes k}$-modules for all $g\in S_k$.

We describe the special case of holomorphic $V$. In this case, also $V^{\otimes k}$ is holomorphic, $g$-rational and there is an up to isomorphism unique irreducible $g$-twisted $V^{\otimes k}$-module $V^{\otimes k}(g)$ for all $g\in\S_k$ (see Theorem~6.4 in \cite{BDM02}).

Let us fix an automorphism $g\in\S_k$ of order $n$ and of cycle shape $\sum_{t\mid n}t^{b_t}$ with $b_t\in\N$, i.e.\ $g$ consists of $b_t$ cycles of length $t$ for each $t\mid n$. Note that $\sum_{t\mid n}tb_t=k$. Now suppose that $V$ has some conformal weight $\rho(V)\in\Z_{\leq 0}$. Then the conformal weight of the \voa{} $V^{\otimes k}$ is
\begin{equation*}
\rho(V^{\otimes k})=k\rho(V)=\rho(V)\sum_{t\mid n}tb_t.
\end{equation*}
On the other hand, from the proof of Theorem~3.9 in \cite{BDM02} we know that the conformal weight of the unique irreducible $g$-twisted $V^{\otimes k}$-module $V^{\otimes k}(g)$ is
\begin{align*}
\rho(V^{\otimes k}(g))&=\rho(V)\sum_{t\mid n}\frac{b_t}{t}+\frac{c}{24}\sum_{t\mid n}b_t\left(t-\frac{1}{t}\right)\\
&=\rho(V^{\otimes k})+\left(\frac{c}{24}-\rho(V)\right)\sum_{t\mid n}b_t\left(t-\frac{1}{t}\right).
\end{align*}
If $V$ is also $C_2$-cofinite, then Zhu's modular invariance implies that the character of $V$
\begin{equation*}
\ch_V(\tau)=Z_{\id,\id}(\tau)=\tr_{V}q^{L_0-c/24}=q^{\rho(V)-c/24}\sum_{k\in\N}\dim(V_{\rho(V)+k})q^k
\end{equation*}
is a modular form for $\SLZ$ of weight 0 and possibly some character which is holomorphic on $\H$. The valence formula implies that $\ch_V(\tau)$ has a pole at the cusp $\i\infty$, which means that $c/24>\rho(V)$. (For a proof of this well-known fact see, e.g.\ Proposition~1.4.10 in \cite{Moe16} where only the special case of $V$ of CFT-type is treated.) Clearly, this shows that
\begin{equation*}
\rho(V^{\otimes k}(g))>\rho(V^{\otimes k})
\end{equation*}
for $g\neq\id$. This proves:
\begin{prop}
The assertion of Conjecture~\ref{conj:main2} (and hence Conjecture~\ref{conj:main}) holds for holomorphic, $C_2$-cofinite \voa{}s of the form $V^{\otimes k}$ and automorphisms in $\mathcal{S}_n\leq\Aut(V^{\otimes k})$.
\end{prop}

\subsection{Lattice \VOA{}s}
Further examples of \voa{}s of which many twisted modules and their conformal weights are explicitly known are lattice \voa{}s \cite{FLM88,Don93}.

Let $L$ be an even, positive-definite lattice, i.e.\ a free abelian group $L$ of finite rank~$\rk(L)$ equipped with a positive-definite, symmetric bilinear form $\langle\cdot,\cdot\rangle\colon L\times L\to\Z$ such that the norm $\langle\alpha,\alpha\rangle/2\in\Z$ for all $\alpha\in L$. Let $\h=L\otimes_\Z\C$ denote the complexification of the lattice $L$.

The lattice \voa{} $V_L$ associated with $L$ is simple, \strat{} and has central charge $c=\rk(L)$. Its irreducible modules are indexed by the discriminant form $L'/L$. Hence, if $L$ is unimodular, i.e.\ if $L'=L$, then $V_L$ is holomorphic.

Let $O(L)$ denote the group of automorphisms (or isometries) of the lattice $L$. The construction of $V_L$ involves a choice of group $2$-cocycle $\eps\colon L\times L\to\{\pm 1\}$. An automorphism $\nu\in O(L)$ and a function $\eta\colon L\to\{\pm 1\}$ satisfying
\begin{equation*}
\eta(\alpha)\eta(\beta)/\eta(\alpha+\beta)=\eps(\alpha,\beta)/\eps(\nu\alpha,\nu\beta)
\end{equation*}
define an automorphism $\hat{\nu}\in O(\hat{L})$ and the automorphisms obtained in this way form the subgroup $O(\hat{L})\leq\Aut(V_L)$ (see, e.g. \cite{FLM88,Bor92}). We call $\hat{\nu}$ a \emph{standard lift} if the restriction of $\eta$ to the fixed-point sublattice $L^\nu\subseteq L$ is trivial. All standard lifts of $\nu$ are conjugate in $\Aut(V_L)$ (see \cite{EMS15}, Proposition~7.1). Let $\hat{\nu}$ be a standard lift of $\nu$ and suppose that $\nu$ has order $m$. Then if $m$ is odd or if $m$ is even and $\langle\alpha,\nu^{m/2}\alpha\rangle$ is even for all $\alpha\in L$, the order of $\hat{\nu}$ is also $m$ and otherwise the order of $\hat{\nu}$ is $2m$, in which case we say $\nu$ exhibits order doubling.

For any \voa{} $V$ of CFT-type $K:=\langle\{\e^{v_0}\,|\,v\in V_1\}\rangle$ defines a subgroup of $\Aut(V)$, called the \emph{inner automorphism group}. Note that since $V$ is of CFT-type, $v_0\omega=0$ for all $v\in V_1$ so that the inner automorphisms preserve the Virasoro vector $\omega$, which is included in the definition of a \voa{} automorphism.

In general, not much is known about the structure of the automorphism group $\Aut(V)$ of a \voa{} $V$. However, for lattice \voa{}s $V_L$ it was show that
\begin{equation*}
\Aut(V_L)=K\cdot O(\hat{L}),
\end{equation*}
$K$ is a normal subgroup of $\Aut(V_L)$ and $\Aut(V_L)/K$ is isomorphic to a quotient group of $O(L)$ (see Theorem~2.1 in \cite{DN99}).

The irreducible $\hat{\nu}$-twisted modules of a lattice \voa{} $V_L$ for standard lifts $\hat{\nu}\in O(\hat{L})$ are described in \cite{DL96,BK04} but the generalisation to non-standard lifts is not difficult (see, e.g.\ Theorem~7.6 in \cite{EMS15}). Combining these results with Section~5 of \cite{Li96} we can explicitly describe the $g$-twisted $V_L$-modules for all automorphisms
\begin{equation*}
g\in K_0\cdot O(\hat{L})
\end{equation*}
with the abelian subgroup
\begin{equation*}
K_0:=\{\e^{(2\pi\i)h_0}\,|\,h\in L\otimes_\Z\Q\}\leq K.
\end{equation*}
Note that the elements of $K_0$ and $O(\hat{L})$ have finite order.

For a lattice \voa{} $V_L$ we can naturally identify the complexified lattice $\h:=L\otimes_\Z\C$ with $\hh:=\{h(-1)\otimes\ee_0\,|\,h\in\h\}\subseteq(V_L)_1$. Now let $g=\sigma_h\hat{\nu}$ with $\sigma_h:=\e^{-(2\pi\i)h_0}$ for some $h\in L\otimes\Q$ and $\hat{\nu}$ some lift of $\nu\in O(L)$. For simplicity we assume that $\hat{\nu}$ is a standard lift. We may always do so since the non-standardness of $\hat{\nu}$ can be absorbed into $\sigma_h$. Moreover, we may assume that $h\in\pi_\nu(\h)$ were $\pi_\nu=\frac{1}{m}\sum_{i=0}^{m-1}\nu^i$ is the projection of $\h$ onto the elements of $\h$ fixed by $\nu$. Indeed, it is shown in Lemma~7.3 of \cite{EMS17} that $\sigma_h\hat{\nu}$ is conjugate to $\sigma_{\pi_\nu(h)}\hat{\nu}$. Since $h\in\pi_\nu(\h)$, $\sigma_h$ and $\hat{\nu}$ commute, allowing us to apply the results in \cite{Li96} to $g=\sigma_h\hat{\nu}$.

In the following let $L$ be unimodular. Then $V_L$ is holomorphic and there is a unique $g$-twisted $V_L$-module $V_L(g)$ for each $g=\sigma_h\hat{\nu}$. Assume that $\nu$ is of order $m$ and has cycle shape $\sum_{t\mid m}t^{b_t}$ with $b_t\in\Z$, i.e.\ the extension of $\nu$ to $\h$ has characteristic polynomial $\prod_{t\mid m}(x^t-1)^{b_t}$. Note that $\sum_{t\mid m}tb_t=\rk(L)=c$. Then the conformal weight of $V_L(\hat{\nu})$ is given by
\begin{equation*}
\rho(V_L(\hat{\nu}))=\frac{1}{24}\sum_{t\mid m}b_t\left(t-\frac{1}{t}\right)+\min_{\alpha\in\pi_{\nu}(L)+h}\langle\alpha,\alpha\rangle/2.
\end{equation*}
The second term is the norm of a shortest vector in the lattice coset $\pi_\nu(L)+h$. Also, $\sum_{t\mid m}b_t\left(t-\frac{1}{t}\right)>0$ for $\nu\neq\id$.

It is clear that $\rho(V_L(\hat{\nu}))\geq 0$ and if $\rho(V_L(\hat{\nu}))=0$, then $\nu=\id$ and $h\in L$, which entails $g=\sigma_h=\id$ since $L=L'$. Hence we have proved:
\begin{prop}
The assertion of Conjecture~\ref{conj:main} holds for \strathol{} \voa{}s of the form $V_L$ for an even, unimodular, positive-definite lattice $L$ and automorphisms in $K_0\cdot O(\hat{L})\leq\Aut(V_L)$.
\end{prop}

\section{Quantum Dimensions and Simple Currents}\label{sec:qdim}
In this section we study the relationship between quantum dimensions and Conjecture~\ref{conj:main} and Theorem~\ref{thm:main}. The case in which Conjecture~\ref{conj:main} holds will differ drastically from the case where only Theorem~\ref{thm:main} is true.

Recall that given a \voa{} $V$ and a $V$-module $W$ and assuming that the characters $\ch_V(\tau)$ and $\ch_W(\tau)$ are well-defined functions on the upper half-plane $\H$ (e.g.\ if $V$ is rational and $C_2$-cofinite and $W$ is an irreducible $V$-module) the \emph{quantum dimension} $\qdim_V(W)$ of $W$ is defined as the limit
\begin{equation*}
\qdim_V(W):=\lim_{y\to 0^+}\frac{\ch_W(iy)}{\ch_V(iy)}
\end{equation*}
if it exists.

Now suppose that $V$ is a \strathol{} \voa{} and $g$ an automorphism of $V$ of finite order. It is shown in Proposition~5.6 in \cite{EMS15} (see also Lemma 4.5.3 in \cite{Moe16} and Proposition~5.4 in \cite{DRX15}) that all irreducible $V^g$-modules are simple currents.\footnote{Recall that a module $U$ of a rational \voa{} $V$ is called \emph{simple current} if the fusion product $U\boxtimes_VW$ is irreducible for every irreducible $V$-module $W$.} For this result the positivity condition for $V^g$ is not needed.

\subsection{Positivity Condition}
If we additionally assume that Conjecture~\ref{conj:main} holds for $V$, then by Proposition~\ref{prop:cor} $V^g$ satisfies the positivity condition. This allows us to apply Proposition~4.17 in \cite{DJX13}, which states that an irreducible $V^g$-module is a simple current if and only if $\qdim_{V^g}(W)=1$. Hence, $\qdim_{V^g}(W)=1$ for all irreducible $V^g$-modules $W$.

\subsection{No Positivity Condition}
We contrast this to the situation where only Theorem~\ref{thm:main} but not necessarily Conjecture~\ref{conj:main} holds. Still all irreducible $V^g$-modules are simple currents but the quantum dimensions might not all be 1.

In the following, let us study the simplest possible case where $g$ is an automorphism of order $n=2$ of the \strathol{} \voa{} $V$. In order to contradict Conjecture~\ref{conj:main} we have to assume that $\rho(V(g))=0$, in which case $\dim(V(g)_0)=1$. There are up to isomorphism exactly four irreducible $V^g$-modules, namely $W^{(0,0)}=V^g$, the eigenspace in $V$ of $g$ corresponding to the eigenvalue 1, $W^{(0,1)}$, the one corresponding to the eigenvalue $-1$, $W^{(1,0)}=V(g)_{\Z}$, the part of $V(g)$ with $L_0$-eigenvalues in $\Z$, and $W^{(1,1)}=V(g)_{\Z+1/2}$, the part with $L_0$-eigenvalues in $\Z+1/2$. Then the cyclic orbifold theory developed in \cite{EMS15,Moe16} finds for the fusion product that
\begin{equation*}
W^{(i,j)}\boxtimes_{V^g}W^{(k,l)}\cong W^{(i+k,j+l)}
\end{equation*}
for all $i,j\in\Z_2$ and the $S$-matrix of $V^g$ (i.e.\ the factors appearing in Zhu's modular invariance result \cite{Zhu96} applied to the characters of the irreducible $V^g$-modules under the transformation $\tau\mapsto S.\tau=-1/\tau$) is given by
\begin{equation*}
\S_{(i,j),(k,l)}=\frac{1}{2}(-1)^{-(kj+il)}.
\end{equation*}
The conformal weights of the irreducible $V^g$-modules obey
\begin{equation*}
\rho(W^{(i,j)})\in\frac{ij}{2}+\Z.
\end{equation*}
More specifically, with the assumptions we made, $\rho(W^{(0,0)})=0$, $\rho(W^{(0,1)})\in\Z_{\geq 1}$, $\rho(W^{(1,0)})=0$ and $\rho(W^{(1,1)})\in1/2+\N.$

Following the same arguments as in the proof of Lemma~4.2 in \cite{DJX13} (but considering the fact that there are two irreducible modules with minimal conformal weight instead of just one) we can show that
\begin{equation*}
\qdim_{V^g}(W^{(i,j)})=\frac{\S_{(i,j),(0,0)}+\S_{(i,j),(1,0)}}{\S_{(0,0),(0,0)}+\S_{(0,0),(1,0)}}=\frac{1+(-1)^j}{1+1}=\delta_{0,j}
\end{equation*}
where we used that $\dim(W^{(0,0)}_0)=\dim(W^{(1,0)}_0)=1$. This differs clearly from the situation where Conjecture~\ref{conj:main} holds.

\bibliographystyle{../../Latex/alpha_noseriescomma}
\bibliography{../../Latex/quellen}{}

\end{document}